\documentclass[review]{elsarticle}
\usepackage{amsthm}
\usepackage{lineno,hyperref}
\usepackage{amssymb,enumerate}
\usepackage{amsmath}
\usepackage{graphics}
\usepackage{xcolor}

\usepackage{geometry}
\newtheorem{thm}{Theorem}[section]
 
 \newtheorem{lem}[thm]{Lemma}
 \newtheorem{prop}[thm]{Proposition}
 \newtheorem{defn}[thm]{Definition}
 \newtheorem{rem}[thm]{Remark}

    \DeclareMathOperator\supp{supp}

	\DeclareMathOperator\meas{meas}

\journal{Journal of \LaTeX\ Templates}

\begin{document}

\begin{frontmatter}

\title{Lifespan of semilinear wave equation with scale invariant dissipation and mass and sub-Strauss power nonlinearity}

\author{Alessandro Palmieri
\footnote{Institute of Applied Analysis, Technical University Bergakademie Freiberg,
Pr\"{u}ferstra{\ss}e 9, 09596, Freiberg, Germany. e-mail: alessandro.palmieri.math@gmail.com}
 \quad
Ziheng Tu
\footnote{School of Data Science, Zhejiang University of Finance and Economics, 310018, Hangzhou, P.R.China.
 e-mail: tuziheng@zufe.edu.cn.}
}

\begin{abstract}

In this paper, we study the blow-up of solutions for semilinear wave equations with scale invariant dissipation and mass in the case in which the model is somehow ``wave-like''. A Strauss type critical exponent is determined as the upper bound for the exponent in the nonlinearity in the main theorems. Two blow-up results are obtained for the sub-critical case and for the critical case, respectively. In both cases, an upper bound lifespan estimate is given.

\end{abstract}


\end{frontmatter}

\section{Introduction and main results}
In present paper, we consider  the  following hyperbolic model
\begin{align}
u_{tt}-\Delta u+\frac{\mu_1}{1+t}u_t+\frac{\mu_2^2}{(1+t)^2}u & =|u|^p  \qquad  \ \  (x,t)\ \in\ \mathbb{R}^n\times[0,\infty), \notag \\
u(0,x)& =\varepsilon \, f(x) \qquad x\ \in\ \mathbb{R}^n, \label{main}\\
u_t(0,x)& =\varepsilon \, g(x) \qquad \, x\ \in\ \mathbb{R}^n. \notag
\end{align}
where $\mu_1,\ \mu^2_2\geqslant 0$ and $\varepsilon>0$ is a  parameter that describes the smallness of initial data. 
The time-dependent coefficients for the damping and for the mass term are chosen in order to have for the corresponding linear equation 
\begin{equation}\label{lin eq scal inv}
u_{tt}-\Delta u+\frac{\mu_1}{1+t}u_t+\frac{\mu_2^2}{(1+t)^2}u=0 
\end{equation}  a scaling property. More precisely, \eqref{lin eq scal inv} is invariant with respect to the so-called hyperbolic transformation
$$\widetilde{u}(t,x)=u(\lambda(1+t)-1,\lambda x) \quad \mbox{with} \ \lambda>0.$$ 

In the last years, \eqref{main}  has been studied in \cite{NPR16,Pal17,PalRei17,PR17vs,DabbPal18,Pal18odd,Pal18even}.

It turns out that the quantity $$\delta:= (\mu_1-1)^2-4\mu_2^2$$ describes the interplay between the damping and the mass term in \eqref{main}. For further considerations on this interplay cf. \cite{NPR16,PR17vs,DabbPal18}.

 Combining the results from \cite{NPR16,Pal17}, it follows that the shift of the Fujita exponent $p_F\Big(n+\tfrac{\mu_1-1-\sqrt{\delta}}2\Big)$ is the critical exponent for \eqref{main} in the case $\delta\geqslant (n+1)^2$, where $p_F(n):=1+\frac{2}{n}$. Therefore, \eqref{main} is ``parabolic-like'' from the point of view of the critical exponent for ``large'' $\delta$. 
 On the other hand, in \cite{PR17vs} it has been proved a blow-up result for $\delta\in (0,1]$ provided that $$1<p\leqslant \max \Big\{p_S(n+\mu_1), p_F\Big(n+\tfrac{\mu_1-1-\sqrt{\delta}}2\Big)\Big\}$$ with the exception of the critical case $p=p_S(n+\mu_1)$ in dimension $n=1$. In the preceding condition $p_S(r)$ denotes the so-called Strauss exponent, that is, the positive root of the quadratic equation 
\begin{equation} \label{qua}
 \gamma(p,r):= 2+(r+1)p-(r-1)p^2= 0 \ \ \mbox{for} \ r>1.
\end{equation} 

Briefly, in \cite{PR17vs} a suitable change of variables allows transforming \eqref{main} in a semilinear wave equation with time-dependent speed of propagation. Hence, a suitable test function, involving the modified Bessel function of the second kind, and Kato's lemma are used. Consequently, we see that for small and positive $\delta$, using the same jargon as before, \eqref{main} seems to be ``wave-like'', at least concerning blow-up results.

The goal of this paper is to enlarge the range of $\delta$ for which a blow-up result can be proved for  $1<p\leqslant p_S(n+\mu_1)$. Furthermore, upper bound estimates for the lifespan of the local (in time) solution of \eqref{main} are derived. 

In the sub-critical case we combine the approach from \cite{YZ06}, in order to determine a lower bound for the integral with respect to spatial variables of the nonlinearity, and an iteration method introduced in \cite{Joh79} for the semilinear free wave equation in dimension $n=3$ and very recently applied to several different models (see \cite{LST18, LT18Scatt,LT18ComNon,TL1709}, for example). 

In the critical case, we adapt the approach of \cite{IS17}, which is based in turn on that one of \cite{ZH14}, in order to include the scale-invariant mass term.

We briefly recall some related background concerning model \eqref{main}. When $\mu_1=\mu_2=0$, this model reduces to the classic semilinear wave equation. In this case, the Strauss exponent  $p_S(n)$ is known to be the critical exponent. We refer to the classical works \cite{Joh79,Gla81-g,Zhou95,LS96,GLS97} for small data global existence results when $p> p_S(n)$, and
\cite{Joh79,Gla81-b,Sid84,Sch85,YZ06,Zhou07} for the blow-up results when $1< p\leqslant p_S(n)$.

When $\mu_2=0$, model \eqref{main} is reduced to the scale invariant damping wave equation which has drawn more and more attention recently. As mentioned in \cite{Wirt04}, such type damping is a threshold betweeen "effective" and "non-effective" dampings. Moreover, the size of $\mu_1$ plays an important role in determining the solution behavior type. In \cite{Abb15,Wakasugi14} it is proved that  $p_F(n)$ is critical for sufficiently large $\mu_1$, while for  $\mu_1<\mu^*:=\frac{n^2+n+2}{n+2}$ in \cite{DLR15,LTW17,IS17,TL1709,TL1711} several blow-up results are given for $p\leqslant  p_S(n+\mu_1)$.
We note that $\mu^*$ satisfies the identity $p_F(n)=p_S(n+\mu^*)$. In particular, in \cite{TL1711} a different test function from that of \cite{IS17} is used in the critical case.  Finally, some global (in time) existence results of small data solutions are proved  for $\mu_1=2$ in \cite{DLR15,DabbLuc15}.

We state now the main results of this paper.
According to \cite{LTW17}, we introduce a notion of energy solution in the following way.
\begin{defn} Let $f\in H^1(\mathbb{R}^n)$ and $g\in L^2(\mathbb{R}^n)$.
We say that $u$ is an energy solution of \eqref{main} on $[0,T)$ if
$$u\in C([0,T),H^1(\mathbb{R}^n))\cap C^1([0,T),L^2(\mathbb{R}^n))\cap L^p_{loc}(\mathbb{R}^n\times[0,T))$$
satisfies
\begin{align} 
&\int_{\mathbb{R}^n}u_t(t,x)\phi(t,x)\,dx-\int_{\mathbb{R}^n}u_t(0,x)\phi(0,x)\,dx - \int_0^t\int_{\mathbb{R}^n} u_t(s,x)\phi_t(s,x) \,dx\, ds\notag \\
& \ +\int_0^t\int_{\mathbb{R}^n}\nabla u(s,x)\cdot\nabla\phi(s,x)\, dx\, ds+\int_0^t\int_{\mathbb{R}^n}\left(\frac{\mu_1 }{1+s} u_t(s,x)+\frac{\mu_2^2}{(1+s)^2}u(s,x)\right)\phi(s,x)\,dx\, ds  \notag \\
& =\int_0^t \int_{\mathbb{R}^n}|u(s,x)|^p\phi(s,x)\,dx \, ds  \label{def}
\end{align}
for any $\phi\in C_0^\infty([0,T)\times\mathbb{R}^n)$ and any $t\in [0,T)$.
\end{defn}
 After a further integration by parts in \eqref{def}, letting $t\rightarrow T$, we find that $u$ fulfills the definition of weak solution of \eqref{main}.
 
Our main results are the following two theorems, where we study the sub-critical case and the critical case, respectively.

\begin{thm}
\label{sub-critical}
Let $n\geqslant 1$ and let $\mu_1, \mu_2^2$ be nonnegative constants such that $\delta\geqslant 0$. Let us consider $p$ satisfying
$1< p <p_S(n+\mu_1).$

Assume that $f\in\ H^1(\mathbb{R}^n)$ and $g\in L^2(\mathbb{R}^n)$ are compactly supported in $B_R:=\{x\in\mathbb{R}^n : |x|\leqslant R\}$ and 
\begin{equation}\label{ini-c1}
f(x)\geqslant 0\ \mbox{and} \ \ g(x)+\tfrac{\mu_1-1-\sqrt{\delta}}{2}f(x)\geqslant 0.
\end{equation}
Let $u$ be an energy solution of \eqref{main} with lifespan $T=T(\varepsilon).$
Then, there exists a constant $\varepsilon_0=\varepsilon_0(f,g,n,p,\mu_1,\mu_2^2,R)$ such that $T(\varepsilon)$ fulfills
$$T(\varepsilon)\leqslant C\varepsilon ^{-2p(p-1)/\gamma(p,n+\mu_1)}$$
for any $0<\varepsilon\leqslant\varepsilon_0$, where $C$ is a positive constant independent of $\varepsilon$.
\end{thm}

\begin{thm}\label{thm critical case} Let $n\geqslant 1$ and let $\mu_1, \, \mu^2_2$ be nonnegative constants such that $0\leqslant \delta <n^2$. Let us consider $p=p_0(n+\mu_1)$. Furthermore, we assume $p>\frac{2}{n-\sqrt{\delta}}$. 

Let $f\in\ H^1(\mathbb{R}^n)$ and $g\in L^2(\mathbb{R}^n$  be nonnegative, not identically zero and  compactly supported in $B_R$ for some $R<1$.

Let us consider  an energy solution $u$ of \eqref{main} with lifespan $T=T(\varepsilon).$ 
  Then, there exists $\varepsilon_0=\varepsilon_0(f,g,n,p,\mu_1,\mu_2^2,R)>0$ such that for any $0<\varepsilon\leqslant \varepsilon_0 $ the solution $u$ blows up in finite time. Furthermore, it holds the following upper bound estimate for the lifespan $T=T(\varepsilon)$ of $u$:
\begin{align}\label{upper bound lifespan}
T(\varepsilon)\leqslant \exp(C\varepsilon^{-p(p-1)})
\end{align}  for some constant $C$ which is independent of $\varepsilon$. 
\end{thm}


The remaining part of the paper is arranged as follows. In Section \ref{Section test func subcrit}, we construct the test function that will be employed in the proof of Theorem \ref{sub-critical}. Furthermore, a lower bound for the $p$ norm of the solution of \eqref{main} is derived. This lower bound will play in turn a fundamental role in the derivation of the lower bound for the time-dependent functional that we will consider in the proof of Theorem \ref{sub-critical}. Similarly, in Section \ref{Section test func crit} we deal with the construction of a different test function, involving Gauss hypergeometric function, and we prove some preliminary results to the proof of Theorem \ref{thm critical case}.  In Sections \ref{Section proof thm subcrit} and  \ref{Section proof thm crit}, we provide the proofs of Theorems \ref{sub-critical} and \ref{thm critical case}, respectively.

\section{Test function and Preliminaries: subcritical case} \label{Section test func subcrit}

Before starting with the construction of the test functions, we recall the definition of the modified Bessel function of the second kind of order $\nu$
$$K_{\nu}(t)=\int_0^{\infty}\exp(-t\cosh z)\cosh(\nu z)dz,\ \ \nu\in \mathbb{R}$$
which is a solution of the equation
$$\bigg(t^2\frac{d^2}{dt^2}+t\frac{d}{dt}-(t^2+\nu^2)\bigg)K_{\nu}(t)=0, \ \ t>0.$$
We collect some important properties concerning $K_\nu(t)$ in the case in which $\nu$ is a real parameter. Interested reader may refer to \cite{ Erdelyi}.

\begin{itemize}
\item The limiting behavior of $K_{\nu}(t)$:
\begin{align}\label{K1}
K_{\nu}(t) & = \sqrt{\dfrac{\pi}{2t}}\, e^{-t}[1+O(t^{-1})]  & \mbox{as} \ t\to \infty.
\end{align}
\item The derivative identity:
\begin{align}
\frac{d}{dt}K_{\nu}(t) & = -K_{\nu+1}(t)+\frac{\nu}{t}K_{\nu}(t). \label{K4} 
\end{align}
\end{itemize}
Firstly, we set the auxiliary function with respect to the time variable
$$\lambda(t):=(t+1)^{\frac{\mu_1+1}{2}}K_{\frac{\sqrt{\delta}}2}(t+1), \ \ \ t\geqslant0.$$
It is clear by direct computation that $\lambda(t)$ satisfies
\begin{equation}
\begin{cases}
\begin{aligned}\label{lamb}
&\bigg((1+t)^2\frac{d^2}{dt^2}-\mu_1(1+t)\frac{d}{dt}+(\mu_1+\mu_2^2-(1+t)^2)\bigg)\lambda(t)=0, \ \ t>0.\\
&\lambda(0)=K_{\frac{\sqrt{\delta}}2}(1),\ \ \ \lambda(\infty)=0.
\end{aligned}
\end{cases}
\end{equation}

Following \cite{YZ06}, let us introduce the function
\begin{equation*}\varphi(x):=
\begin{cases}
\int_{\mathbb{S}^{n-1}}e^{x\cdot \omega}d\omega \ \ &\mbox{for}\ \ n\geqslant2,\\
e^x+e^{-x} \ \ &\mbox{for}\ \ n=1.
\end{cases}
\end{equation*}
 The function $\varphi$  satisfies $$\Delta\varphi(x)=\varphi(x)$$ and the asymptotic estimate
\begin{equation}
\varphi(x)\sim C_n|x|^{-\frac{n-1}2}e^{|x|} \ \ \ \mbox{as}\ \ \ |x|\rightarrow\infty.
\end{equation}
We define the test function for the sub-critical case
$$\psi(t,x):=\lambda(t)\varphi(x).$$

In the following lemma, we derive a lower bound for $\int_{\mathbb{R}^n}|u(x,t)|^pdx$.
\begin{lem}
Let us assume $f,\ g$ such that $\supp f,\  \supp g \subset B_R$ for some $R>0$ and \eqref{ini-c1} is fulfilled. Then, a local energy solution $u$ satisfies $$\mbox{\rm supp}\ u\ \subset\{(t,x)\in  [0,T)\times\mathbb{R}^n: |x|\leqslant t+R \}  $$ and
there exists a large $T_0$, which is independent of $f,\ g$ and $\varepsilon$, such that for any $t>T_0$ and $p>1$, it holds
\begin{equation}\label{Priori}
\int_{\mathbb{R}^n}|u(t,x)|^pdx\geqslant C_1\varepsilon^p(1+t)^{n-1-\frac{n+\mu_1-1}2 p},
\end{equation}
where {$C_1=C_1(f,g,\varphi,p,R)$}.
\end{lem}

\begin{proof}
Define the functional
$$F(t):=\int_{\mathbb{R}^n}u(t,x)\psi(t,x)dx$$
with $\psi(t,x)=\lambda(t)\varphi(x)$ defined as above. Then, by H\"{o}lder inequality, we have
\begin{equation}\int_{\mathbb{R}^n}|u(t,x)|^pdx\geqslant |F(t)|^p\bigg(\int_{|x|\leqslant t+R}\psi^{p'}(t,x)dx\bigg)^{-(p-1)}.\label{factor}\end{equation}
{The next step is to determine a lower bound for $|F(t)|$ and an upper bound for the integral $\int_{|x|\leqslant t+R}\psi^{p'}(t,x)dx$,  respectively.}
From the definition of energy solution, we have
\begin{eqnarray*}
&&\int_0^t\int_{\mathbb{R}^n}u_{tt}\psi \, dxds-\int_{0}^t\int_{\mathbb{R}^n}u\Delta\psi \, dxds\\
&&+\int_0^t\int_{\mathbb{R}^n}\Big(\partial_s\Big(\frac{\mu_1}{1+s}\psi u\Big)-\partial_s\Big(\frac{\mu_1}{1+s}\psi\Big)u+\frac{\mu_2^2}{(1+s)^2}\psi u\Big) dxds=\int_0^t\int_{\mathbb{R}^n}|u|^p\psi \,dxds.
\end{eqnarray*}
Applying integration by parts and $\Delta\varphi(x)=\varphi$, we obtain:
\begin{eqnarray*}&&\int_0^t\int_{\mathbb{R}^n}u_{tt}\psi \, dxds+\int_{0}^t\int_{\mathbb{R}^n}u\varphi\Big(-\lambda+\frac{\mu_1+\mu_2^2}{(1+s)^2}\lambda-\frac{\mu_1}{1+s}\lambda'\Big)dxds\\
&&+\int_{\mathbb{R}^n}\frac{\mu_1}{1+s}\psi \, udx\bigg|_0^t =\int_0^t\int_{\mathbb{R}^n}|u|^p\psi \, dxds.
\end{eqnarray*}
Simplifying the above equation by plugging \eqref{lamb} gives
$$\int_0^t\int_{\mathbb{R}^n}u_{tt}\psi \, dxds-\int_{0}^t\int_{\mathbb{R}^n}u\varphi\lambda''  dxds+\int_{\mathbb{R}^n}\frac{\mu_1}{1+s}\psi u \, dx\bigg|_0^t =\int_0^t\int_{\mathbb{R}^n}|u|^p\psi \, dxds.$$
Hence, a further integration by parts leads to
$$\int_{\mathbb{R}^n}\Big(u_t\psi-u\psi_t+\frac{\mu_1}{1+s}u\psi\Big) dx\bigg|_0^t=\int_0^t\int_{\mathbb{R}^n}|u|^p\psi \, dxds.$$
As the righthand side integral is positive, we obtain
$$F'(t)+\bigg(\frac{\mu_1}{1+t}-2\frac{\lambda'(t)}{\lambda(t)}\bigg)F(t)\geqslant\varepsilon \int_{\mathbb{R}^n}\bigg(g(x)\lambda(0)+(\mu_1\lambda(0)-\lambda'(0))f(x)\bigg)\varphi(x)\, dx.$$
Using \eqref{K4}, we have
\begin{align*}
\lambda'(t) & =\tfrac{\mu_1+1}2(1+t)^{\frac{\mu_1-1}2}K_{\frac{\sqrt{\delta}}2}(1+t)+(1+t)^{\frac{\mu_1+1}2}K'_{\frac{\sqrt{\delta}}2}(1+t)\nonumber\\
& =\tfrac{\mu_1+1}2(1+t)^{\frac{\mu_1-1}2}K_{\frac{\sqrt{\delta}}2}(1+t)+(1+t)^{\frac{\mu_1+1}{2}}\left(-K_{\frac{\sqrt{\delta}}2+1}(1+t)+\tfrac{\sqrt{\delta}}{2(1+t)}K_{\frac{\sqrt{\delta}}2}(1+t)\right)\nonumber\\
& =\tfrac{\mu_1+1+\sqrt{\delta}}2(1+t)^{\frac{\mu_1-1}2}K_{\frac{\sqrt{\delta}}2}(1+t)-(1+t)^{\frac{\mu_1+1}{2}}K_{\frac{\sqrt{\delta}}2+1}(1+t),
\end{align*}
Also,
\begin{align*}
\lambda'(0)&=\tfrac{\mu_1+1+\sqrt{\delta}}2K_{\frac{\sqrt{\delta}}2}(1)-K_{\frac{\sqrt{\delta}}2+1}(1), \\
\mu_1\lambda(0)-\lambda'(0) & =\tfrac{\mu_1-1-\sqrt{\delta}}2K_{\frac{\sqrt{\delta}}2}(1)+K_{\frac{\sqrt{\delta}}2+1}(1).
\end{align*}

 Consequently,
\begin{align*}
&g(x)\lambda(0)+(\mu_1\lambda(0)-\lambda'(0))f(x)=K_{\frac{\sqrt{\delta}}2}(1)\left(g(x)+\tfrac{\mu_1-1-\sqrt{\delta}}{2}f(x)	\right)+K_{\frac{\sqrt{\delta}}2+1}(1)f(x).
\end{align*}

Denote $$C_{f,g}:=\int_{\mathbb{R}^n}\bigg(g(x)\lambda(0)+(\mu_1\lambda(0)-\lambda'(0))f(x)\bigg)\varphi(x)\, dx,$$
then, since we assume compactly supported and satisfying \eqref{ini-c1} $f$ and $g$ , $C_{f,g}$ is finite and positive.
We thus conclude that $F$ satisfies the differential inequality
$$F'(t)+\bigg(\frac{\mu_1}{1+t}-2\frac{\lambda'(t)}{\lambda(t)}\bigg)F(t)\geqslant\varepsilon \, C_{f,g}.$$
Multiplying $\frac{(1+t)^{\mu_1}}{\lambda^2(t)}$ on two sides and then integrating over $[0,t]$, we derive
$$F(t)\geqslant\varepsilon \, C_{f,g}\frac{\lambda^2(t)}{(1+t)^{\mu_1}}\int_0^t\frac{(1+s)^{\mu_1}}{\lambda^2(s)}\, ds.$$
Inserting $\lambda(t)=(1+t)^{\frac{\mu_1+1}{2}}K_{\frac{\sqrt{\delta}}2}(1+t)$, we obtain the lower bound of $F$
\begin{equation} F(t)\geqslant\varepsilon \,  C_{f,g}\int_0^t\frac{(1+t)K^2_{\frac{\sqrt{\delta}}{2}}(1+t)}{(1+s)K^2_{\frac{\sqrt{\delta}}{2}}(1+s)}\, ds \,  \geqslant 0.\label{G1}\end{equation}

The second factor in the right-hand side of \eqref{factor} can be estimated in standard way (cf. \cite[estimate (2.5)]{YZ06}).
\begin{align}
\int_{|x|\leqslant t+R}\psi^{p'}(t,x)\, dx & \leqslant \lambda^\frac{p}{p-1}(t)\int_{|x|\leqslant t+R}\varphi^{p'}(x)\, dx\notag\\
&\leqslant C_{\varphi,R}(1+t)^{n-1+\big(\frac{\mu_1+1}{2}-\frac{n-1}{2}\big)\frac{p}{p-1}}e^{\frac{p}{p-1}(t+R)}K^{\frac{p}{p-1}}_{\frac{\sqrt{\delta}}2}(1+t),\label{deno}
\end{align}
where $C_{\varphi,R}$ is a suitable positive constant.

Combing the estimate \eqref{G1}, \eqref{deno} and \eqref{factor}, we now have

\begin{align*}\int_{\mathbb{R}^n}|u(x,t)|^pdx 
&\geqslant C_{f,g}^{p}C_{\varphi,R}^{1-p}\varepsilon^p(1+t)^{p -(n-1)(p-1)-\big(\frac{\mu_1+1}{2}-\frac{n-1}{2}\big)p} e^{-p(t+R)} K^p_{\frac{\sqrt{\delta}}2}(1+t)\\ & \qquad \qquad \times\bigg(\int_0^t\frac{ds}{(1+s)K^2_{\frac{\sqrt{\delta}}2}(1+s)}\bigg)^p \\
&\geqslant C_{f,g}^{p}C_{\varphi,R}^{1-p}e^{p(1-R)}\varepsilon^p(1+t)^{(2-n-\mu_1)\frac p2+(n-1)} e^{-p(1+t)} K^p_{\frac{\sqrt{\delta}}2}(1+t)\\ & \qquad \qquad \times\bigg(\int_0^t\frac{ds}{(1+s)K^2_{\frac{\sqrt{\delta}}2}(1+s)}\bigg)^p.
\end{align*}
Since \eqref{K1}, then for sufficient large $T_0$ (which is independent of $f,\ g,\ \varepsilon$) and $t>T_0$, we have
$$K^p_{\frac{\sqrt{\delta}}2}(1+t)\sim \bigg(\frac{\pi}{2(1+t)}\bigg)^\frac p2 e^{-p(t+1)}$$
and
\begin{align*}\int_0^t\frac1{(1+s)K^2_{\frac{\sqrt{\delta}}2}(1+s)}ds &\geqslant \int_{\frac t2}^t\frac{2}\pi e^{2(1+s)}ds \\
&= \frac1\pi\big(e^{2(1+t)}-e^{2+t}\big)\geqslant\frac1{2\pi }e^{2(1+t)}.
\end{align*}
Consequently,
$$\int_{\mathbb{R}^n}|u(x,t)|^pdx\geqslant C_{1}\varepsilon^p(1+t)^{\frac p2(1-n-\mu_1)+(n-1)}\ \ \mbox{for}\ \ t>T_0,$$
where $C_1:= 2^{-p}C_{f,g}^{p}C_{\varphi,R}^{1-p}e^{p(1-R)}\pi^{-p}.$ This concludes the proof.
\end{proof}

\section{Proof of Theorem \ref{sub-critical}  } \label{Section proof thm subcrit}
Let $u$ be an energy solution of \eqref{main} on $[0,T)$ and define
$$G(t):=\int_{\mathbb{R}^n}u(t,x)\,dx.$$
Choosing a  $\phi=\phi(s,x)$ in \eqref{def} that satisfies $\phi\equiv1$ in $\{(x,s)\in  [0,t]\times \mathbb{R}^n :|x|\leqslant s+R\}$, we obtain
\begin{eqnarray*}&&\int_{\mathbb{R}^n}u_t(t,x)\,dx-\int_{\mathbb{R}^n}u_t(0,x)\,dx+\int_0^tds\int_{\mathbb{R}^n}\bigg(\frac{\mu_1 u_t(s,x)}{1+s}+\frac{\mu^2_2u(s,x)}{(1+s)^2}\bigg)dx\\
&&=\int_0^tds\int_{\mathbb{R}^n}|u(s,x)|^pdx
\end{eqnarray*}
which means that
\begin{equation*}
G'(t)-G'(0)+\int_0^t\frac{\mu_1 G'(s)}{1+s}\,ds+\int_0^t\frac{\mu_2^2G(s)}{(1+s)^2}\,ds=\int_0^tds\int_{\mathbb{R}^n}|u(s,x)|^pdx.
\end{equation*}
Since all functions in this equation aside from $G'(t)$ are differentiable in $t$, $G'(t)$ is differentiable in $t$ as well. Hence, we have
\begin{equation}\label{G-Dyn}
G''(t)+\frac{\mu_1}{1+t}G'(t)+\frac{\mu^2_2}{(1+t)^2}G(t)=\int_{\mathbb{R}^n}|u(t,x)|^pdx.
\end{equation}
Consider the quadratic equation
$$r^2-(\mu_1-1)r+\mu_2^2=0.$$
As $\delta\geqslant0$, there exit two real roots,
$$r_1=\frac{\mu_1-1-\sqrt{\delta}}{2},\ \ \ r_2=\frac{\mu_1-1+\sqrt{\delta}}{2}.$$
{Clearly}, if $\mu_1>1$ then both $r_1$ and $r_2$ are positive. {Else, if} $0\leqslant\mu_1<1$, both $r_1$ and $r_2$ are negative. When $\mu_1=1$ then $\mu_2=0$ as $\delta\geqslant0$, {and} hence $r_1=r_2=0$.
Moreover, in whatever situation  $$r_{1,2}+1>0.$$
We may rewrite \eqref{G-Dyn} as
$$\Big(G'(t)+\frac{r_1}{1+t}G(t)\Big)'+\frac{r_2+1}{1+t}\Big(G'(t)+\frac{r_1}{1+t}G(t)\Big)=\int_{\mathbb{R}^n}|u(t,x)|^pdx.$$
Multiplying by $(1+t)^{r_2+1}$ and integrating over $[0,t]$, we obtain
$$(1+t)^{r_2+1}\Big(G'(t)+\frac{r_1}{1+t}G(t)\Big)-\Big(G'(0)+r_1G(0)\Big)=\int_0^t(1+s)^{r_2+1}ds\int_{\mathbb{R}^n}|u|^pdx.$$
Using \eqref{ini-c1}, we have
\begin{equation}\label{key}
G'(t)+\frac{r_1}{1+t}G(t)>(1+t)^{-r_2-1}\int_0^t(1+s)^{r_2+1}ds\int_{\mathbb{R}^n}|u|^pdx.
\end{equation}
Multiplying the above inequality by $(1+t)^{r_1}$ and integrating over $[0,t]$ gives
\begin{equation*}
(1+t)^{r_1}G(t)-G(0)>\int_0^t(1+\tau)^{r_1-r_2-1}d\tau\int_0^\tau(1+s)^{r_2+1} ds\int_{\mathbb{R}^n}|u(s,x)|^pdx.
\end{equation*}
{By the positivity assumption on $f$, we have
\begin{eqnarray}
 G(t)&\geqslant&\int_0^t\bigg(\frac{1+\tau}{1+t}\bigg)^{r_1}d\tau\int_0^\tau \bigg(\frac{1+s}{1+\tau}\bigg)^{r_2+1} ds\int_{\mathbb{R}^n}|u(s,x)|^pdx. \label{iter1}
\end{eqnarray} Furthermore, using H\"{o}lder inequality and the compactness of the support of solution with respect to $x$, we get from \eqref{iter1}
\begin{eqnarray}
  G(t) &\geqslant&C_0\int_0^t\bigg(\frac{1+\tau}{1+t}\bigg)^{r_1}d\tau\int_0^\tau \bigg(\frac{1+s}{1+\tau}\bigg)^{r_2+1}(1+s)^{n(1-p)}|G(s)|^pds\label{iter2}
\end{eqnarray}
where   is used in second inequality and
$$C_0:=(\meas(B_1))^{1-p}R^{-n(p-1)}>0.$$}
At this moment, we are ready to prove {Theorem \ref{sub-critical}}. {We shall apply an iteration method based on lower bound estimates \eqref{Priori}, \eqref{iter1} and \eqref{iter2}.}
\begin{proof}[Proof of Theorem \ref{sub-critical}]
Plugging \eqref{Priori} into \eqref{iter1}, we have for $t>T_0$,
\begin{eqnarray}
G(t)&\geqslant&\int_0^t\bigg(\frac{1+\tau}{1+t}\bigg)^{r_1}d\tau\int_0^\tau\bigg(\frac{1+s}{1+\tau}\bigg)^{r_2+1} C_1\varepsilon^p(1+s)^{n-1-\frac{n+\mu_1-1}{2}p}ds\nonumber\\
&\geqslant&C_1\varepsilon^p(1+t)^{-r_1}\int_{T_0}^t(1+\tau)^{r_1-r_2-1}d\tau\int_{T_0}^\tau(1+s)^{n+r_2-(n+\mu_1-1)\frac p2}ds\nonumber\\
&\geqslant&C_1\varepsilon^p(1+t)^{-r_1}\int_{T_0}^t(1+\tau)^{r_1-r_2-1-(n+\mu_1-1)\frac p2}d\tau\int_{T_0}^\tau(1+s)^{n+r_2}ds\nonumber\\
&\geqslant&C_1\varepsilon^p(1+t)^{-r_2-1-(n+\mu_1-1)\frac p2}\int_{T_0}^td\tau\int_{T_0}^\tau(s-T_0)^{n+r_2}ds.\label{r_1<r_2}
\end{eqnarray}
That is,
\begin{equation}\label{j=1}
G(t)\geqslant C_2\, \varepsilon^p(1+t)^{-r_2-1-(n+\mu_1-1)\frac p2}(t-T_0)^{n+r_2+2}\ \ \mbox{for}\ \ t>T_0,
\end{equation}
where $C_2=\frac{C_1}{(n+r_2+1)(n+r_2+2)}$. Notice that, in \eqref{r_1<r_2} we may simply use the property
$$r_1-r_2-1-(n+\mu_1-1)\frac p2\leqslant 0.$$

Now we begin our iteration argument. Assume that
\begin{equation}\label{iter-assu}
G(t)\geqslant D_j(1+t)^{-a_j}(t-T_0)^{b_j}\ \ \ \mbox{for}\ \ t>T_0,\ \ j=1,2,3 \dots
\end{equation}
with positive constants $D_j,\ a_j$ and $b_j$ {to be} determined later. {From \eqref{j=1} it follows that \eqref{iter-assu}} is true for $j=1$ with
\begin{equation}\label{series1}D_1=C_2\varepsilon^p,\ \ a_1=r_2+1+(n+\mu_1-1)\frac p2,\ \ \ b_1=n+r_2+2.\end{equation}
Plugging \eqref{iter-assu} into \eqref{iter2}, we have for $t>T_0$
\begin{align}
G(t)&\geqslant C_0\,(1+t)^{-r_1}\int_{T_0}^t(1+\tau)^{r_1-r_2-1}d\tau \int_{T_0}^\tau(1+s)^{r_2+1+n(1-p)}D_j^p(1+s)^{-pa_j}(s-T_0)^{pb_j}ds\notag\\
&\geqslant  C_0\, D_j^p(1+t)^{-r_2-1-n(p-1)-pa_j}\int_{T_0}^t\int_{T_0}^\tau(s-T_0)^{r_2+1+pb_j}dsd\tau\label{r_1r2}\\
&\geqslant \frac{C_0D_j^p}{(r_2+pb_j+2)(r_2+pb_j+3)}(1+t)^{-r_2-1-n(p-1)-pa_j}(t-T_0)^{r_2+pb_j+3}.\notag
\end{align}
where in \eqref{r_1r2}, we utilize
$$r_1-r_2-1-n(p-1)-pa_j\leqslant0.$$
So \eqref{iter-assu} is true if the sequences $\{D_j\}$, $\{a_j\}$, $\{b_j\}$ fulfill
{\begin{align} \label{cond Dj+1}
D_{j+1} & \geqslant\frac{C_0}{(r_2+pb_j+3)^2}D_j^p,\\
 a_{j+1} &=r_2+1+n(p-1)+pa_j, \quad  b_{j+1}=r_2+3+pb_j. \label{def aj+1 and bj+1}
\end{align}}
It follows from \eqref{series1} and \eqref{def aj+1 and bj+1} that for $j=1,2,3\dots$
\begin{eqnarray}
a_j&=&\bigg(a_1+n+\frac{r_2+1}{p-1}\bigg)p^{j-1}-\bigg(n+\frac{r_2+1}{p-1}\bigg)\nonumber\\
&=&\alpha p^{j-1}-\bigg(n+\frac{r_2+1}{p-1}\bigg),\label{a_j}\\
b_j&=&\bigg(b_1+\frac{r_2+3}{p-1}\bigg)p^{j-1}-\frac{r_2+3}{p-1}\nonumber\\
&=&\beta p^{j-1}-\frac{r_2+3}{p-1},\label{b_j}
\end{eqnarray}
where we denote the positive constants $$\alpha=
r_2+1+(n+\mu_1-1)\frac p2+n+\frac{r_2+1}{p-1},\ \ \beta=n+r_2+2+\frac{r_2+3}{p-1}.$$
{Using \eqref{def aj+1 and bj+1} and \eqref{b_j}, we get}
$$b_{j+1}=r_2+3+pb_j<p^j\beta.$$
{Therefore, we obtain from the previous inequality and  \eqref{cond Dj+1}}
$$D_{j+1}\geqslant C_3\frac{D^p_j}{p^{2j}}$$
where $$C_3{=\frac{C_0}{\beta^2}}=\frac{C_0}{\Big(n+r_2+2+\frac{r_2+3}{p-1}\Big)^2}.$$
Hence,
\begin{eqnarray*}
\log D_j&\geqslant& p\log D_{j-1}-2(j-1)\log p+\log C_3\\
&\geqslant& p^2\log D_{j-2}-2(p(j-2)+(j-1))\log p+(p+1)\log C_3\\
&\geqslant&\cdots\\
&\geqslant&p^{j-1}\log D_1-2\log p\sum_{k=1}^{j-1}kp^{j-1-k}+\log C_3\sum_{k=1}^{j-1}p^k.
\end{eqnarray*}
{Using an inductive argument, the following formulas can be shown:}
$$\sum_{k=1}^{j-1}kp^{j-1-k}=\frac{1}{p-1}\bigg(\frac{p^j-1}{p-1}-j\bigg)$$ and
$$\sum_{k=1}^{j-1}p^k=\frac{p-p^j}{1-p},$$
which yield
\begin{eqnarray*}
\log D_j&\geqslant& p^{j-1}\log D_1-\frac{2\log p}{p-1}\bigg(\frac{p^j-1}{p-1}-j\bigg)+\log C_3\frac{p-p^j}{1-p}\\
&=&p^{j-1}\bigg(\log D_1-\frac{2p\log p}{(p-1)^2}+\frac{p\log C_3}{p-1}\bigg)+\frac{2\log p}{p-1}j+\frac{2\log p}{(p-1)^2}+\frac{p\log C_3}{1-p}.
\end{eqnarray*}
Consequently, for $j>\left[\frac{p\log C_3}{2\log p}-\frac{1}{p-1}\right]+1$ it holds
\begin{equation}\label{D_j}D_j\geqslant\exp\{p^{j-1}(\log D_1-S_p(\infty))\}\end{equation}
with
$$ S_p(\infty):=\frac{2p\log p}{(p-1)^2}-\frac{p\log C_3}{p-1}.$$
Inserting \eqref{a_j}, \eqref{b_j} and \eqref{D_j} into \eqref{iter-assu} gives
\begin{eqnarray}
G(t)&\geqslant&\exp\big(p^{j-1}(\log D_1-S_p(\infty))\big)(1+t)^{-\alpha p^{j-1}+n+\frac{r_2+1}{p-1}}(t-T_0)^{\beta p^{j-1}-\frac{r_2+3}{p-1}}\nonumber\\
&\geqslant&\exp\big(p^{j-1}J(t)\big)(1+t)^{n+\frac{r_2+1}{p-1}}(t-T_0)^{-\frac{r_2+3}{p-1}},\label{contr}
\end{eqnarray}
where
$$J(t):=\log D_1-S_p(\infty)-\alpha\log(1+t)+\beta\log(t-T_0).$$
For $t>2T_0+1$, we have
\begin{eqnarray*}
J(t)&\geqslant& \log D_1-S_p(\infty)-\alpha\log(2t-2T_0)+\beta\log(t-T_0)\\
&\geqslant&\log D_1-S_p(\infty)+(\beta-\alpha)\log(t-T_0)-\alpha\log2\\
&=&\log(D_1\cdot(t-T_0)^{\beta-\alpha})-S_p(\infty)-\alpha\log2.
\end{eqnarray*}
Note that
$$\beta-\alpha=b_1-a_1-n+\frac2{p-1}=\frac{p+1}{p-1}-(n+\mu_1-1)\frac p2=\frac{\gamma(p,n+\mu_1)}{2(p-1)}.$$
Thus, if $$t>\max\bigg\{T_0+\bigg(\frac{e^{(S_p(\infty)+\alpha\log2)+1}}{C_2\varepsilon^p}\bigg)^{2(p-1)/\gamma(p,n+\mu_1)},2T_0+1\bigg\},$$
then, we get $J(t)>1$, and this in turn gives $G(t)\rightarrow\infty$ by taking $j\rightarrow\infty$ in \eqref{contr}. Therefore, there exists a sufficiently small $\varepsilon_0>0$ such that for any $\varepsilon<\varepsilon_0$ we obtain the desired upper bound,
$$T\leqslant C_4\varepsilon^{-\frac{2p(p-1)}{\gamma(p,n+\mu_1)}}$$
with
$$C_4:=\left(\frac{e^{(S_p(\infty)+\alpha\log2)+1}}{C_2}\right)^{2(p-1)/\gamma(p,n+\mu_1)}.$$
This completes our proof of Theorem \ref{sub-critical}.
\end{proof}

%

\section{Test function and preliminaries: critical case} \label{Section test func crit}

In this section and in the next one, we adapt the approach from \cite{IS17}, with  the purpose to include the scale-invariant mass term.

Firstly, let us construct a suitable solution of the adjoint equation of  \eqref{lin eq scal inv} in 
$Q_1:=\{(t,x)\in  [0,\infty) \times \mathbb{R}^n: |x|< 1+t \}$. In other terms, we look for a function $\Phi=\Phi(t,x)$ which solves
\begin{equation}\label{lin eq PHI}
\partial_t^2 \Phi  -\Delta \Phi -\partial_t \Big(\frac{\mu_1}{1+t} \, \Phi \Big)+\frac{\mu_2^2}{(1+t)^2} \,\Phi =0 \quad \mbox{for any} \ \  (t,x)\in Q_1.
\end{equation}

\begin{prop}\label{Prop rel PHI and psi} Let $\beta$ be a real parameter. Let us make the following ansatz:
\begin{equation}\label{kind PHI}
 \Phi_\beta(t,x):= (1+t)^{-\beta+1} \psi_\beta\bigg( \frac{|x|^2}{(1+t)^2}\bigg),
 \end{equation} where $\psi_\beta \in \mathcal{C}^2([0,1)).$
Then, $\Phi_\beta$ solves \eqref{lin eq PHI} if and only if $\psi_\beta$ solves
\begin{equation}\label{lin eq psi}
z(1-z)\psi^{\prime\prime}_\beta (z)+\big(\tfrac{n}{2}-\big(\beta+\tfrac{1}{2}+\tfrac{\mu_1}{2}\big) z\big)\psi^{\prime}_\beta (z) -\big( \tfrac{\beta(\beta+\mu_1-1)+\mu_2^2}{4}\big) \psi_\beta(z) =0.
\end{equation}
\end{prop}



\begin{proof}
 For the sake of brevity we introduce the notation $z:=\frac{|x|^2}{(1+t)^2}$.
By straightforward computations, it follows
\begin{align*}
\partial_t \Phi_\beta(t,x)  
  &=  (-\beta+1) (1+t)^{-\beta} \psi_\beta (z)-2(1+t)^{-\beta} z\,  \psi^{\prime}_\beta (z), \\
\partial_t ^2\Phi_\beta(t,x)  
  &=  (\beta-1) \beta (1+t)^{-\beta-1} \psi_\beta (z)+4(\beta-1)(1+t)^{-\beta-1} z\,  \psi^{\prime}_\beta (z) \\ & \quad +4(1+t)^{-\beta-1} z^2 \, \psi^{\prime \prime}_\beta (z) +6(1+t)^{-\beta-1} z\, \psi^{\prime}_\beta (z) , 
\end{align*}
 and 
\begin{align*}
\Delta\Phi_\beta(t,x)  
&= 2n(1+t)^{-\beta-1} \psi^{\prime}_\beta (z) +4(1+t)^{-\beta-1} z \, \psi^{\prime\prime}_\beta (z) .
\end{align*}

Plugging the previous relations, we obtain the following identity:
\begin{align*}
 \partial_t^2 \Phi_\beta - & \Delta \Phi_\beta-\partial_t \Big(\frac{\mu_1}{1+t} \, \Phi_\beta\Big)+\frac{\mu_2^2}{(1+t)^2}\,  \Phi_\beta 
\\ &= (1+t)^{-\beta-1}\Big(4z(z-1)\psi^{\prime\prime}_\beta (z)+((4(\beta-1)+6+2\mu_1) z-2n)\psi^{\prime}_\beta (z) \\ & \qquad\qquad\qquad\quad+ (\beta(\beta-1)-\mu_1(-\beta+1)+\mu_1+\mu_2^2) \psi_\beta(z) \Big).
\end{align*}
Also, $\Phi_\beta$ solves \eqref{lin eq PHI} if and only if $\psi_\beta$ is a solution to \eqref{lin eq psi}.
\end{proof}

If we find $a,b$ such that 
\begin{equation}\label{condition on a,b}
a+b+1=\beta+\tfrac{1}{2}+\tfrac{\mu_1}{2}\, , \quad
ab  = \tfrac{\beta(\beta+\mu_1-1)+\mu_2^2}{4}\, ,
\end{equation} then, \eqref{lin eq psi} coincides with the hypergeometric equation with parameters $(a,b\, ;\tfrac{n}{2})$, namely,
\begin{align*}
z(1-z)\psi^{\prime\prime}_\beta (z)+\big(\tfrac{n}{2}-(a+b+1) z\big)\psi^{\prime}_\beta (z) -a b \psi_\beta(z) =0.
\end{align*}

 Hence, whether $a,\ b$ fulfill \eqref{condition on a,b}, we can choose the Gauss hypergeometric function with parameters $(a,b\, ;\tfrac{n}{2})$  as solution to the above equation, i.e.,
\begin{equation} \label{def psi beta}
\psi_\beta(z) :=F(a,b\, ;\tfrac{n}{2}\, ;z) =\sum_{k=0}^\infty \frac{(a)_k (b)_k}{(n/2)_k} \frac{z^k}{k!},
\end{equation}  provided that $|z|<1$ or, equivalently, $(t,x)\in Q_1$. In \eqref{def psi beta} the so-called Pochhammer's symbol $(m)_k$ is defined by $$(m)_k =\begin{cases} 1 & \mbox{if} \ k=0, \\ \prod_{j=1}^k (m+j-1) & \mbox{if} \ k\geqslant 0.\end{cases}$$ 

It is actually possible to choose $a,b$ satisfying \eqref{condition on a,b}.
Indeed, the quadratic equations 
\begin{align*}
r^2-\big(\beta+\tfrac{\mu_1}{2}-\tfrac{1}{2}\big)r + \tfrac{\beta(\beta+\mu_1-1)+\mu_2^2}{4}=0
\end{align*} has an  independent of $\beta$ and nonnegative discriminant due to the assumption $\delta\geqslant 0$. Let us introduce
\begin{align}
a & := \tfrac{\beta}{2} +\tfrac{\mu_1-1}{4}+ \tfrac{\sqrt{\delta}}{4}, \label{definition a} \\
b & := \tfrac{\beta}{2} +\tfrac{\mu_1-1}{4}- \tfrac{\sqrt{\delta}}{4}. \label{definition b} 
\end{align}
Then, $a$ and $b$ fulfill \eqref{condition on a,b}.

\begin{defn} Let $a,b$ be defined by \eqref{definition a} and \eqref{definition b}, respectively. We introduce the following function
\begin{align}\label{definition PSI}
\Phi_\beta= \Phi_{\beta}(t,x\, ;\mu_1,\mu_2) & \hphantom{:}= (1+t)^{-\beta+1} \psi_{\beta}\Big( \tfrac{|x|^2}{(1+t)^2}\Big) \notag\\ & := (1+t)^{-\beta+1} F\Big(a,b\, ;\tfrac{n}{2}\, ; \tfrac{|x|^2}{(1+t)^2}\Big) \quad \mbox{for} \ (t,x)\in Q_1.
\end{align} 
\end{defn}
 According to Proposition \ref{Prop rel PHI and psi} $\Phi_\beta$ solves \eqref{lin eq PHI} in $Q_1$. The next step is to provide the asymptotic behavior of $\psi_\beta $ and $\psi'_\beta$.
 
\begin{lem} The following estimates are satisfied:
\begin{itemize}
\item[\rm{(i)}] if $\frac{\sqrt{\delta}-\mu_1+1}{2}<\beta <\frac{n-\mu_1+1}{2}$, then, there exists  $C'=C'(\beta,n,\mu_1,\mu_2)>1$ such that for any $z\in [0,1)$ it holds
\begin{align}\label{asymp psi beta}
1\leqslant \psi_\beta (z)\leqslant C'\, ; 
\end{align}
\item[\rm{(ii)}]  if $\beta >\frac{n-\mu_1-1}{2}$, then, there exists $C''=C''(\beta, n, \mu_1,\mu_2)>1$ such that for any $z\in [0,1)$ it holds
\begin{align} \label{asymp psi' beta}
 \tfrac{1}{C''} (1-\sqrt{z})^{\frac{n-\mu_1-1}{2}-\beta} \leqslant \big|\psi'_\beta (z)\big|\leqslant C''  (1-\sqrt{z})^{\frac{n-\mu_1-1}{2}-\beta} \, . 
 \end{align}
\end{itemize}
\end{lem} 

\begin{proof}
(i) The assumption on  $\beta$ implies that $a,b>0 $ and $a+b<\frac{n}{2}$. Since $\psi_\beta=F(a,b,\, ; \tfrac{n}{2}\, ;\cdot)$, \eqref{asymp psi beta} follows immediately by \cite[Section 15.4 (ii), formula 15.4.20]{OLBC10}. \\
(ii) Because of $\psi'_\beta=\tfrac{2ab}{n} F(a+1,b+1\, ; \tfrac{n}{2}+1\, ; \cdot)$, the assumption on $\beta$ and \cite[Section 15.4 (ii), formula 15.4.23]{OLBC10} imply \eqref{asymp psi' beta}.
\end{proof}
 
Before proving Theorem \ref{thm critical case}, we derive some preliminary lemmas. First of all, we introduce the following functionals 
\begin{align}
\mathcal{G}_\beta(t) & := \int_{\mathbb{R}^n} |u(t,x)|^p\, \Phi_\beta(t,x \, ; \mu_1,\mu_2) \, dx\, , \label{definition G beta}\\
\mathcal{H}_\beta(t) & := \int_0^t (t-s)(1+s)\,\mathcal{G}_\beta(s) \, ds\, ,  \label{definition H beta}  \\
\mathcal{J}_\beta(t) & := \int_0^t (2+s)^{-3}\,\mathcal{H}_\beta(s) \, ds\, , \label{definition J beta}
\end{align}
where $\beta\in \Big(\frac{\sqrt{\delta}-\mu_1+1}{2}, \frac{n-\mu_1+1}{2}\Big)$ and $t\geqslant 0$.
We remark that $\delta$ should be smaller than $n^2$ in order to get a nonempty range for $\beta$.

\begin{rem}\label{rmk lifespans} From \eqref{asymp psi beta} it follows that $G_\beta(t) \approx (1+t)^{1-\beta}\| u(t,\cdot)\|_{L^p(\mathbb{R}^n)}^p$. Hence, if we prove that $\mathcal{J}_\beta$ blows up in finite time, then, in turn, $\mathcal{H}_\beta$ blows up in finite time and $G_\beta(t)$ as well. Due to the previous relation, we get hence that the lifespan of $\mathcal{J}_\beta$ is an upper bound for the lifespan $T$ of the energy solution solution $u$ of \eqref{main}.
\end{rem}

\begin{lem} \label{lemma relation G beta and J beta} For any $\beta\in \Big(\frac{\sqrt{\delta}-\mu_1+1}{2}, \frac{n-\mu_1+1}{2}\Big)$ and $t\geqslant 0$ it holds
\begin{align*}
(1+t)^2\mathcal{J}_\beta(t) \leqslant \frac{1}{2} \int_0^t (t-s)^2 \mathcal{G}_\beta(s) \, ds.
\end{align*}
\end{lem}

\begin{proof}
Differentiating  twice \eqref{definition H beta}, we have
\begin{align}
\mathcal{H}'_\beta (t) = \int_0^t (1+s) \,\mathcal{G}_\beta(s) \, ds \, , \quad \label{derivative H beta}
\mathcal{H}''_\beta (t) = \ (1+t) \,\mathcal{G}_\beta(t).
\end{align}
Then, by using integration by parts, since $\mathcal{H}_\beta(0)=\mathcal{H}'_\beta(0)=0$, we get
\begin{align*}
\int_0^t (t-s)^2 \mathcal{G}_\beta(s) \, ds &= \int_0^t (t-s)^2 (1+s)^{-1} \mathcal{H}''_\beta(s) \, ds = \int_0^t \partial_s^2[ (t-s)^2 (1+s)^{-1}] \mathcal{H}_\beta(s) \, ds \\
 &= 2 \int_0^t (1+s)^{-3}(1+t)^2 \mathcal{H}_\beta(s) \, ds \geqslant 
 2\,  (1+t)^2 \mathcal{J}_\beta(t),
\end{align*} which is exactly the desired inequality.
\end{proof}

\begin{lem}\label{lemma exact identity}
Let us assume $(f,g)\in H^1(\mathbb{R}^n)\times L^2(\mathbb{R}^n)$ nonnegative, not identically zero, compactly supported such that $$\supp(f), \ \supp(g) \subset B_{R} \ \mbox{and} \quad R<1.$$

Let $u$ be a solution of \eqref{main}. Then, for every $\beta\in \Big(\frac{\sqrt{\delta}-\mu_1+1}{2}, \frac{n-\mu_1+1}{2}\Big)$ such that $\beta\geqslant1-\mu_1$ and $t\geqslant 0$ the following identity holds
\begin{align}
&\varepsilon E_{0,\beta}(f)+\varepsilon E_{1,\beta}(f,g)\, t +\int_0^t (t-s) \mathcal{G}_\beta(s) \, ds \notag \\
& \quad =\int_{\mathbb{R}^n} u(t,x)\Phi_\beta(t,x) \, dx +\int_0^t (1+s)^{-\beta} \int_{\mathbb{R}^n} 
u(s,x) \,\widetilde{\psi}_\beta\Big(\tfrac{|x|^2}{(1+s)^2}\Big)\, dx \, ds,
\label{fundamental identity}
\end{align}
where
\begin{align}\label{definition E1}
E_{1,\beta}(f,g)& := \int_{\mathbb{R}^n}  \Big(g(x)\psi_\beta(|x|^2)+f(x) \big((\beta-1+\mu_1)\psi_\beta(|x|^2) +2|x|^2\psi'_\beta(|x|^2) \big) \Big) \, dx  ,\\
E_{0,\beta}(f)& := \int_{\mathbb{R}^n}  f(x)\psi_\beta(|x|^2)\, dx  
 \label{definition E2}
\end{align} are positive quantities and
\begin{equation}\label{definition psi tilde}
\widetilde{\psi}_\beta(z): =(2\beta+\mu_1-2)\psi_\beta(z)+4 z\,\psi'_\beta(z).
\end{equation}
\end{lem}

\begin{proof}
Due to the property of finite speed of propagation for solutions of strictly hyperbolic equations, for the solution $u$ of the semilinear Cauchy problem \eqref{main} we have $\supp u(t,\cdot) \subset B_{R+t}$ for any $t\geqslant 0$, which implies $\supp u \subset Q_1$, as $R<1$.

For the sake of brevity, we will denote simply $ \Phi_\beta(t,x\, ;\mu_1,\mu_2)\equiv \Phi_\beta(t,x)$.
Then, using \eqref{lin eq PHI}, we have
\begin{align}
\mathcal{G}_\beta(t)
& =\int_{\mathbb{R}^n} \Big(u_{tt}(t,x)-\Delta u(t,x)+\tfrac{\mu_1}{1+t}u_t(t,x)+\tfrac{\mu_2^2}{(1+t)^2}u(t,x)\Big)\Phi_\beta(t,x) \, dx \notag\\ & \quad-  \int_{\mathbb{R}^n} u(t,x)\Big(\partial_t^2 \Phi_\beta(t,x)-\Delta \Phi_\beta(t,x) -\partial_t\Big(\tfrac{\mu_1}{1+t}\Phi_\beta(t,x)\Big)+\tfrac{\mu_2^2}{(1+t)^2}\Phi_\beta(t,x)\Big)\, dx \notag\\
& =\int_{\mathbb{R}^n} \Big(u_{tt}(t,x)+\tfrac{\mu_1}{1+t}u_t(t,x)\Big)\Phi_\beta(t,x) \, dx \notag\\
& \quad-  \int_{\mathbb{R}^n} u(t,x)\Big(\partial_t^2 \Phi_\beta(t,x) -\partial_t\Big(\tfrac{\mu_1}{1+t}\Phi_\beta(t,x)\Big)\Big)\, dx \notag\\
& =\frac{d}{dt}\bigg(\int_{\mathbb{R}^n} \big(u_{t}(t,x)\Phi_\beta(t,x)-u(t,x) \partial_t \Phi_\beta(t,x)\big) \, dx +\tfrac{\mu_1}{1+t} \int_{\mathbb{R}^n} u(t,x)\Phi_\beta(t,x)\, dx \bigg),\label{alternative expression of G beta}
\end{align} where in the second equality we used Green's second identity (the boundary integrals with respect to $x$ disappear due to the support property of $u$). 

Since $\partial_t \Phi_\beta(t,x)=(1+t)^{-\beta} \big((-\beta+1)\psi_\beta(z)-2z\psi'_\beta(z)\big),$ then,
\begin{align*}
\int_{\mathbb{R}^n}  &\big(u_{t}(0,x)\Phi_\beta(0,x)-u(0,x) \partial_t \Phi_\beta(0,x)\big) \, dx +\mu_1 \int_{\mathbb{R}^n} u(0,x)\Phi_\beta(0,x)\, dx  \\
&= \varepsilon \int_{\mathbb{R}^n}  \Big(g(x)\psi_\beta(|x|^2)-f(x) \big((-\beta+1)\psi_\beta(|x|^2) -2|x|^2\psi'_\beta(|x|^2) \big)+\mu_1 f(x)\psi_\beta (|x|^2) \Big) \, dx \\
&= \varepsilon \int_{\mathbb{R}^n}  \Big(g(x)\psi_\beta(|x|^2)+f(x) \big((\beta-1+\mu_1)\psi_\beta(|x|^2) +2|x|^2\psi'_\beta(|x|^2) \big) \Big) \, dx= \varepsilon E_{\beta,1}(f,g).
\end{align*} 


Since $\psi_\beta(|x|^2)=F(a,b\, ; \tfrac{n}{2}\, ; |x|^2)\geqslant 1$ and
\begin{align*}
\psi'_\beta(|x|^2)=F'\big(a,b\, ; \tfrac{n}{2}\, ; |x|^2\big) = \tfrac{2 ab}{n} F\big(a+1,b+1\, ; \tfrac{n}{2}+1\, ; |x|^2\big)>0
\end{align*} for $|x|<1$ and we required $\beta\geqslant 1-\mu_1$ in the assumptions, then, it results $E_{\beta,1}(f,g)>0$, as $f$ and $g$ are nonnegative.

Integrating \eqref{alternative expression of G beta} over $[0,t]$, we obtain
\begin{align*}
& \varepsilon E_{\beta,1}(f,g)+\int_0^t \mathcal{G}_\beta(s) \, ds \\ 
& \quad= \int_{\mathbb{R}^n} \big(u_{t}(t,x)\Phi_\beta(t,x)-u(t,x) \partial_t \Phi_\beta(t,x)\big) \, dx +\tfrac{\mu_1}{1+t} \int_{\mathbb{R}^n} u(t,x)\Phi_\beta(t,x)\, dx \\ 
& \quad= \frac{d}{dt}\bigg(\int_{\mathbb{R}^n} u(t,x)\Phi_\beta(t,x) \, dx\bigg)-2 \int_{\mathbb{R}^n} 
u(t,x) \partial_t \Phi_\beta(t,x) \, dx+\tfrac{\mu_1}{1+t} \int_{\mathbb{R}^n} u(t,x)\Phi_\beta(t,x)\, dx \\
& \quad= \frac{d}{dt}\bigg(\int_{\mathbb{R}^n} u(t,x)\Phi_\beta(t,x) \, dx\bigg)+(1+t)^{-\beta} \int_{\mathbb{R}^n} 
u(t,x) \,\widetilde{\psi}_\beta\Big(\tfrac{|x|^2}{(1+t)^2}\Big)\, dx,
\end{align*} where $\widetilde{\psi}_\beta$ is given by \eqref{definition psi tilde}.

A further integration over $[0,t]$ and Fubini's theorem provide
\begin{align*}
\varepsilon E_{\beta,1}(f,g)\, t+\int_0^t \int_0^\tau \mathcal{G}_\beta(s) \, ds 
&   =  \varepsilon E_{\beta,1}(f,g)\, t+\int_0^t (t-s)\,\mathcal{G}_\beta(s) \, ds \, d\tau \\ 
&   = \int_{\mathbb{R}^n} u(t,x)\Phi_\beta(t,x) \, dx -\varepsilon\int_{\mathbb{R}^n} f(x)\psi_\beta(|x|^2) \, dx \\ 
& \quad +\int_0^t (1+s)^{-\beta} \int_{\mathbb{R}^n} 
u(s,x) \,\widetilde{\psi}_\beta\Big(\tfrac{|x|^2}{(1+s)^2}\Big)\, dx \, ds,
\end{align*} that is, \eqref{fundamental identity}.
\end{proof}

\begin{lem} \label{lemma fund inequalities}
Let us assume $(f,g)\in H^1(\mathbb{R}^n)\times L^2(\mathbb{R}^n)$ nonnegative, not identically zero, compactly supported such that $$\supp(f), \ \supp(g) \subset B_{R} \ \mbox{and} \quad R<1.$$

Let  $\mu_1, \ \mu_2$ be nonnegative constants  such that $0\leqslant\delta<n^2$ and let $p=p_0(n+\mu_1)$ be such that $p>\frac{2}{n-\sqrt{\delta}}$.
\begin{enumerate}
\item[\rm{(i)}] Let $q>p$ satisfy 
\begin{align}\label{beta q cond}
\beta_q =\tfrac{n-\mu_1+1}{2}-\tfrac{1}{q}. 
\end{align}
Then,
\begin{align*}
& \varepsilon E_{0,\beta_q}+\varepsilon E_{1,\beta_q}\, t +\int_0^t (t-s) \,\mathcal{G}_{\beta_q}(s) \, ds \\
& \quad \leqslant C_1 \bigg(  (1+t)^{\frac{n}{p'}-\beta_q+1}\| u(t,\cdot)\|_{L^p(\mathbb{R}^n)}+\int_0^t  (1+s)^{\frac{n}{p'}-\beta_p} \| u(s,\cdot)\|_{L^p(\mathbb{R}^n)} \, ds  \bigg).
\end{align*}
\item[\rm{(ii)}] Let $p=q$. If $\beta_p$ is defined by \eqref{beta q cond}, then,
\begin{align*}
& \int_0^t (t-s) \,\mathcal{G}_{\beta_p}(s) \, ds \\
& \quad \leqslant C_1 \bigg(  (1+t)^{\frac{n}{p'}-\beta_p+1}\| u(t,\cdot)\|_{L^p(\mathbb{R}^n)}+\int_0^t  (1+s)^{\frac{n}{p'}-\beta_p}\big(\log(2+s)\big)^{\frac{1}{p'}} \| u(s,\cdot)\|_{L^p(\mathbb{R}^n)} \, ds  \bigg).
\end{align*}
\end{enumerate}
Here the constant $C_1>0$ does not depend on $(t,x)$ and $u$.
\end{lem}

\begin{rem}\label{rmk beta q range} Let us point out that the condition $p>\frac{2}{n-\sqrt{\delta}}$ implies $$\beta_q \geqslant \beta_p > \tfrac{\sqrt{\delta}-\mu_1+1}{2}.$$ for $q\geqslant p$.
 Moreover, the condition $\beta_p\geqslant 1-\mu_1$ is always true for $p=p_0(n+\mu_1)$. Indeed, $\beta_p\geqslant 1-\mu_1$ is equivalent to require $$\tfrac{1}{p}\leqslant \tfrac{n+\mu_1-1}{2}.$$
Besides, $p$ solves the quadratic equation  $\gamma(p,n+\mu_1)=0$.
Therefore,
\begin{align*}
\tfrac{1}{p}=\tfrac{n+\mu_1-1}{2}p-\tfrac{n+\mu_1+1}{2}\leqslant \tfrac{n+\mu_1-1}{2} \quad\ \mbox{if and only if} \quad p \leqslant \tfrac{2(n+\mu_1)}{n+\mu_1-1}.
\end{align*} Being $p=p_0(n+\mu_1)$, a straightforward calculation shows that the last inequality is always fulfilled by nonnegative parameters $\mu_1$.
\end{rem}

\begin{proof}
By \eqref{fundamental identity}, using again the finite speed of propagation property, we may write
\begin{align}
\varepsilon E_{0,\beta_q}(f)+\varepsilon E_{1,\beta_q}(f,g)\, t +\int_0^t (t-s) \,\mathcal{G}_{\beta_q}(s) \, ds = \mathcal{I}_{\beta_q,1}(t)+\int_0^t\mathcal{I}_{\beta_q,2}(s)\, ds, \label{fundamental estimate q}
\end{align} where
\begin{align*}
\mathcal{I}_{\beta_q,1}(t) & := \int_{B_{R+t}} u(t,x)\Phi_{\beta_q}(t,x) \, dx, \\ 
\mathcal{I}_{\beta_q,2}(t) &: =  (1+t)^{-\beta_q} \int_{B_{R+t}} 
u(t,x) \,\widetilde{\psi}_{\beta_q}\Big(\tfrac{|x|^2}{(1+t)^2}\Big)\, dx.
\end{align*}
Let us point out explicitly that, according to Remark \ref{rmk beta q range} we have that the assumptions on $p$ imply $\beta_q\in\Big(\tfrac{\sqrt{\delta}-\mu_1+1}{2}, \tfrac{n-\mu_1+1}{2}\Big)$ and $\beta_q\geqslant 1-\mu_1$. For this reason, we may use Lemma \ref{lemma exact identity} in order to derive \eqref{fundamental estimate q}.
For $\beta_q\in\Big(\tfrac{\sqrt{\delta}-\mu_1+1}{2}, \tfrac{n-\mu_1+1}{2}\Big)$, as the hypergeometric function in \eqref{definition PSI} is uniformly bounded, we can estimate $\Phi_{\beta_q}(t,x)\approx (1+t)^{-\beta_q+1}$ according to  \eqref{asymp psi beta}. Therefore, if we denote by $p'$ the conjugate exponent of $p$, H\"older inequality implies
\begin{align}
\mathcal{I}_{\beta_q,1}(t) &\leqslant \bigg(\int_{B_{R+t}} |u(t,x)|^p \, dx\bigg)^{\frac{1}{p}} \bigg(\int_{B_{R+t}} \Phi_{\beta_q}(t,x)^{p'} \, dx\bigg)^{\frac{1}{p'}} \notag \\ &\leqslant C_1 (1+t)^{\frac{n}{p'}-\beta_q+1} \|u(t,\cdot)\|_{L^p(\mathbb{R}^n)},\label{estimate I beta 1}
\end{align} where throughout this proof $C_1=C_1(n,p,\mu_1,\mu_2,\beta, R)>0$ is a suitable constant that may change from line to line.

Let us estimate now the term $\mathcal{I}_{\beta_q,2}(s)$.
We remark that for $\beta_q$ as in \eqref{beta q cond}, then, $\beta_q> \tfrac{n-\mu_1-1}{2}$, since it is $q>1$. Therefore, in order to estimate $\psi'_{\beta_q}$ we may use \eqref{asymp psi' beta}. As we underlined in the previous case, due to the assumption on $\beta_q$, the function $\psi_{\beta_q}$ is uniformly bounded. Thus, in \eqref{definition psi tilde} the dominant term as $z\to 1^-$ is the derivative.  Hence,
\begin{align*}
|\widetilde{\psi}_{\beta_q}(z)|\leqslant C_1  (1-\sqrt{z})^{\frac{n-\mu_1-1}{2}-\beta_q} \qquad \mbox{for} \ z \in[0,1).
\end{align*}
Consequently, by using H\"older inequality, for $\mathcal{I}_{\beta,2}(s)$ we get
\begin{align*}
\mathcal{I}_{\beta_q,2}(s) & \leqslant (1+s)^{-\beta_q}  \bigg(\int_{B_{R+s}} |u(s,x)|^p \, dx\bigg)^{\frac{1}{p}} \bigg(\int_{B_{R+s}} |\widetilde{\psi}_{\beta_q}(s,x)|^{p'}\, dx\bigg)^{\frac{1}{p'}} \\
& \leqslant C_1 (1+s)^{-\beta_q}   \bigg(\int_{B_{R+s}} \Big(1-\tfrac{|x|}{1+s}\Big)^{(\frac{n-\mu_1-1}{2}-\beta_q)p'}\, dx\bigg)^{\frac{1}{p'}} \|u(s,\cdot)\|_{L^p(\mathbb{R}^n)}\\
& = C_1 (1+s)^{-\beta_q}   \bigg(\int_{B_{R+s}} \Big(1-\tfrac{|x|}{1+s}\Big)^{(\frac{1}{q}-1)p'}\, dx\bigg)^{\frac{1}{p'}} \|u(s,\cdot)\|_{L^p(\mathbb{R}^n)}\\
& = C_1 (1+s)^{-\beta_q}   \bigg(\int_{B_{R+s}} \Big(1-\tfrac{|x|}{1+s}\Big)^{-\frac{p'}{q'}}\, dx\bigg)^{\frac{1}{p'}} \|u(s,\cdot)\|_{L^p(\mathbb{R}^n)},
\end{align*} where $q'$ denotes the conjugate exponent of $q$.

Using polar coordinates, we get 
\begin{align*}
\int_{B_{R+s}} \Big(1-\tfrac{|x|}{1+s}\Big)^{-\frac{p'}{q'}}\, dx & = \omega_{n-1}\int_0^{R+s} \Big(1-\tfrac{r}{1+s}\Big)^{-\frac{p'}{q'}}r^{n-1}\, dr \\& = \omega_{n-1}(1+s)^n \int_0^{\frac{R+s}{1+s}} (1-\rho)^{-\frac{p'}{q'}}\rho^{n-1}\, d\rho,
\end{align*} where $\omega_{n-1}$ is the measure of the unitary sphere $\partial B_1$. Also,
\begin{align*}
\mathcal{I}_{\beta_q,2}(s) & \leqslant C_1 (1+s)^{-\beta_q+\frac{n}{p'}}   \bigg(\int_0^{\frac{R+s}{1+s}} (1-\rho)^{-\frac{p'}{q'}}\rho^{n-1}\, d\rho \bigg)^{\frac{1}{p'}} \|u(s,\cdot)\|_{L^p(\mathbb{R}^n)} \\
& \leqslant C_1 (1+s)^{-\beta_q+\frac{n}{p'}}   \bigg(\int_0^{\frac{R+s}{1+s}} (1-\rho)^{-\frac{p'}{q'}}\, d\rho \bigg)^{\frac{1}{p'}} \|u(s,\cdot)\|_{L^p(\mathbb{R}^n)}\\
& \leqslant C_1 (1+s)^{-\beta_q+\frac{n}{p'}}   \|u(s,\cdot)\|_{L^p(\mathbb{R}^n)}
\begin{cases} (1-\tfrac{R+s}{1+s})^{-\frac{1}{q'}+\frac{1}{p'}} & \mbox{if} \ q>p , \\
\big(-\log(1-\tfrac{R+s}{1+s})\big)^{\frac{1}{p'}} & \mbox{if} \ q=p ,
\end{cases} \\
& \leqslant 
\begin{cases} C_1 (1+s)^{-\beta_q+\frac{n}{p'}+\frac{1}{q'}-\frac{1}{p'}} \|u(s,\cdot)\|_{L^p(\mathbb{R}^n)} & \mbox{if} \ q>p , \\
 C_1 (1+s)^{-\beta_p+\frac{n}{p'}}\big(\log(2+s)\big)^{\frac{1}{p'}}   \|u(s,\cdot)\|_{L^p(\mathbb{R}^n)} & \mbox{if} \ q=p.
\end{cases}
\end{align*} 

Since $-\beta_q+\frac{n}{p'}+\frac{1}{q'}-\frac{1}{p'}=\frac{n}{p'}+1-\tfrac{n-\mu_1+1}{2}-\frac{1}{p'}=\frac{n}{p'}-\beta_p$, integrating  $\mathcal{I}_{\beta,2}(s)$ over $[0,t]$, we find
\begin{align}
\int_0^t \mathcal{I}_{\beta_q,2}(s) \, ds & \leqslant C_1
\begin{cases} \displaystyle{ \int_0^t (1+s)^{\frac{n}{p'}-\beta_p} \|u(s,\cdot)\|_{L^p(\mathbb{R}^n)} \, ds } & \mbox{if} \ q>p , \\
\displaystyle{ \int_0^t (1+s)^{\frac{n}{p'}-\beta_p}\big(\log(2+s)\big)^{\frac{1}{p'}}   \|u(s,\cdot)\|_{L^p(\mathbb{R}^n)}\, ds } & \mbox{if} \ q=p.
\end{cases} \label{estimate I beta 2}
\end{align}

Due to the assumptions on $(f,g)$, we have $E_{0,\beta_q}(f)>0$ and $E_{1,\beta_q}(f,g)>0$. Then, combining \eqref{estimate I beta 1} and \eqref{estimate I beta 2},  from \eqref{fundamental estimate q} we get the desired estimates in the cases $q>p$ and $q=p$. 
\end{proof}

\section{Proof of Theorem \ref{thm critical case} } \label{Section proof thm crit}

\begin{proof}
Let us consider $\beta_{p+\sigma}=\tfrac{n+\mu_1-1}{2}-\tfrac{1}{p+\sigma}$ for $p=p_0(n+\mu_1)$, where $\sigma$ is a positive constant. Being $\beta_q$ increasing with respect to $q$, if we assume $p>\tfrac{2}{n-\sqrt{\delta}}$, then, $\beta_{p+\sigma}>\beta_p>\tfrac{\sqrt{\delta}-\mu_1+1}{2}$ and we can apply Lemma \ref{lemma fund inequalities}.

From Lemma \ref{lemma fund inequalities} (i) it follows
\begin{align*}
& \varepsilon E_{0,\beta_{p+\sigma}}(f)+\varepsilon E_{1,\beta_{p+\sigma}}(f,g)\, t \\ 
& \quad \leqslant C_1 \bigg(  (1+t)^{\frac{n}{p'}-\beta_{p+\sigma}+1}\| u(t,\cdot)\|_{L^p(\mathbb{R}^n)}+\int_0^t  (1+s)^{\frac{n}{p'}-\beta_p} \| u(s,\cdot)\|_{L^p(\mathbb{R}^n)} \, ds  \bigg)\\ 
& \quad \leqslant C_1 \bigg(  (1+t)^{\frac{n+1-\beta_p}{p'}+\beta_{p}-\beta_{p+\sigma}}(\mathcal{G}_{\beta_p}(t))^{\frac{1}{p}}+\int_0^t  (1+s)^{\frac{n+1-\beta_p}{p'}-1} (\mathcal{G}_{\beta_p}(s))^{\frac{1}{p}} \, ds  \bigg).
\end{align*}
Let us underline that  $p=p(n+\mu_1)$ implies $\frac{n+1-\beta_p}{p'}=1+\frac{1}{p}$. Indeed, 
\begin{align*}
\tfrac{n+1-\beta_p}{p'}&=\Big(n+1-\Big(\tfrac{n-\mu_1+1}{2}-\tfrac{1}{p}\Big)\Big)\Big(1-\tfrac{1}{p}\Big) 
 =\tfrac{1}{p^2}\Big(\tfrac{n+\mu_1+1}{2}p+1\Big)(p-1) 
 \\ &=\tfrac{1}{p^2}\Big(\tfrac{n+\mu_1+1}{2}p^2-\tfrac{n+\mu_1+1}{2}p+p-1\Big)
   =\tfrac{1}{p^2}\Big(p^2+p-\underbrace{\tfrac{\gamma(p,n+\mu_1)}{2}}_{=0}\Big)
 = \tfrac{p+1}{p}.
\end{align*}
Then,  integrating the preceding inequality over $[0,t]$ and applying Fubini's theorem and H\"older inequality, we arrive at 
\begin{align*}
& \varepsilon E_{0,\beta_{p+\sigma}}(f)\, t+\tfrac{\varepsilon}{2} E_{1,\beta_{p+\sigma}}(f,g)\, t^2 \\ 
& \quad \leqslant C_1 \bigg(  \int_0^t (1+s)^{1+\frac{1}{p}+\beta_{p}-\beta_{p+\sigma}}(\mathcal{G}_{\beta_p}(s))^{\frac{1}{p}} \, ds+\int_0^t  (t-s)(1+s)^{\frac{1}{p}} (\mathcal{G}_{\beta_p}(s))^{\frac{1}{p}} \, ds  \bigg) \\ 
& \quad \leqslant C_1 \bigg(  \int_0^t (1+s) \mathcal{G}_{\beta_p}(s) \, ds\bigg)^{\frac{1}{p}}\Bigg[\Bigg(  \int_0^t (1+s)^{p'+(\beta_{p}-\beta_{p+\sigma})p'}\, ds \bigg)^{\frac{1}{p'}} +\bigg(\int_0^t  (t-s)^{p'}  \, ds  \bigg)^{\frac{1}{p'}}  \Bigg] \\ 
& \quad \leqslant C_2 \bigg(  \int_0^t (1+s) \mathcal{G}_{\beta_p}(s) \, ds\bigg)^{\frac{1}{p}}(1+t)^{1+\frac{1}{p'}}.
\end{align*}
From \eqref{derivative H beta}, we get 
\begin{align*}
 \varepsilon E_{0,\beta_{p+\sigma}}(f)\, t+\tfrac{\varepsilon}{2} E_{1,\beta_{p+\sigma}}(f,g)\, t^2  \leqslant C_2 \mathcal{H}'_{\beta_p}(t)^{\frac{1}{p}}(1+t)^{1+\frac{1}{p'}},
\end{align*} which implies for $t\geqslant1$
\begin{align*}
\mathcal{H}'_{\beta_p}(t)  & \geqslant C_2^{-p}\varepsilon^p \Big(E_{0,\beta_{p+\sigma}}(f)\, t+\tfrac{1}{2} E_{1,\beta_{p+\sigma}}(f,g)\, t^2  \Big)^p (1+t)^{1-2p} \notag\\
 & \geqslant C_3 \varepsilon^p  (1+t).
\end{align*}
As the functional $H_{\beta_p}$ is nonnegative, from the  previous inequality we get for $t\geqslant 2$
\begin{align}
\mathcal{H}_{\beta_p}(t) \geqslant \int_1^t \mathcal{H}'_{\beta_p}(s) \, ds  \geqslant C_3 \varepsilon^p \int_1^t (1+s) \, ds \geqslant  C_4 \varepsilon^p  (1+t)^2. \label{lower bound H beta}
\end{align}

By using Lemma \ref{lemma fund inequalities} (ii), due to $E_{0,\beta_{p}}(f), E_{1,\beta_{p}}(f,g)\geqslant 0$  we have 
\begin{align*}
& \int_0^t (t-s) \,\mathcal{G}_{\beta_p}(s) \, ds \\
& \quad \leqslant C_1 \bigg(  (1+t)^{\frac{n}{p'}-\beta_p+1}\| u(t,\cdot)\|_{L^p(\mathbb{R}^n)}+\int_0^t  (1+s)^{\frac{n}{p'}-\beta_p}\big(\log(2+s)\big)^{\frac{1}{p'}} \| u(s,\cdot)\|_{L^p(\mathbb{R}^n)} \, ds  \bigg) \\
& \quad \leqslant C_1 \bigg(  (1+t)^{\frac{n+1-\beta_p}{p'}}(\mathcal{G}_{\beta_p}(t))^{\frac{1}{p}}+\int_0^t  (1+s)^{\frac{n+1-\beta_p}{p'}-1}\big(\log(2+s)\big)^{\frac{1}{p'}} (\mathcal{G}_{\beta_p}(s))^{\frac{1}{p}} \, ds  \bigg).
\end{align*}
Integrating over $[0,t]$ and using again the equality $\frac{n+1-\beta_p}{p'}=1+\frac{1}{p}$, H\"older inequality and \eqref{derivative H beta}, we find
\begin{align*}
& \frac{1}{2}\int_0^t (t-s)^2 \,\mathcal{G}_{\beta_p}(s) \, ds \\
& \quad \leqslant C_1 \bigg( \int_0^t (1+s)^{1+\frac{1}{p}}(\mathcal{G}_{\beta_p}(s))^{\frac{1}{p}} \, ds+\int_0^t (t-s) (1+s)^{\frac{1}{p}}\big(\log(2+s)\big)^{\frac{1}{p'}} (\mathcal{G}_{\beta_p}(s))^{\frac{1}{p}} \, ds  \bigg) \\
& \quad \leqslant C'_2 \bigg( \int_0^t (1+s)\mathcal{G}_{\beta_p}(s)\, ds \bigg)^{\frac{1}{p}} \Bigg[\bigg( \int_0^t (1+s)^{p'}\, ds\bigg)^{\frac{1}{p'}}+\bigg(\int_0^t (t-s)^{p'} \log(2+s) \, ds  \bigg)^{\frac{1}{p'}} \Bigg] \\
& \quad \leqslant C'_3 \bigg( \int_0^t (1+s)\mathcal{G}_{\beta_p}(s)\, ds \bigg)^{\frac{1}{p}} (1+t)^{1+\frac{1}{p'}} \big(\log(2+t)\big)^{\frac{1}{p'}} \\ 
& \quad  \leqslant C'_3 \big(\mathcal{H}'_{\beta_p}(t)\big)^{\frac{1}{p}} (2+t)^{\frac{2p-1}{p}} \big(\log(2+t)\big)^{\frac{1}{p'}}.
\end{align*}

From Lemma \ref{lemma relation G beta and J beta}, we get
\begin{align*}
(1+t)^2 \mathcal{J}_{\beta_p}(t) \leqslant C'_3 \big(\mathcal{H}'_{\beta_p}(t)\big)^{\frac{1}{p}} (2+t)^{\frac{2p-1}{p}} \big(\log(2+t)\big)^{\frac{1}{p'}},
\end{align*} and, hence, for $t\geqslant 2$ we have
\begin{align}\label{lower bound H'}
C'_4  \big(\log(2+t)\big)^{1-p} \big(\mathcal{J}_{\beta_p}(t)\big)^p \leqslant \mathcal{H}'_{\beta_p}(t)(2+t)^{-1}.
\end{align}
By the definition of $\mathcal{J}_{\beta_p}$, it follows immediately $(2+t)^3\mathcal{J}'_{\beta_p}(t)=\mathcal{H}_{\beta_p}(t)$ which implies
\begin{align*}
(2+t)^3\mathcal{J}''_{\beta_p}(t)+3(2+t)^2\mathcal{J}'_{\beta_p}(t)=\mathcal{H}'_{\beta_p}(t).
\end{align*}

Combining the previous identity with \eqref{lower bound H'}, we have
\begin{align}
(2+t)^2\mathcal{J}''_{\beta_p}(t)+3(2+t)\mathcal{J}'_{\beta_p}(t) \geqslant C'_4  \big(\log(2+t)\big)^{1-p} \big(\mathcal{J}_{\beta_p}(t)\big)^p.\label{fund rel J}
\end{align}
Moreover, from  \eqref{lower bound H beta}, we get for $t\geqslant 2$ and for a suitable constant $c_0>0$
\begin{align}
\mathcal{J}_{\beta_p}(t)&\geqslant C'_4 \varepsilon^p \int_0^t (2+s)^{-3} (1+s)^2 \, ds \geqslant c_0 \, \varepsilon^p \log(2+t) \label{lower bound J beta}, \\ 
\mathcal{J}'_{\beta_p}(t)& \geqslant C'_4 \varepsilon^p (2+t)^{-3} (1+t)^2 \geqslant c_0 \, \varepsilon^p (2+t)^{-1}. \label{lower bound J' beta}
\end{align}

Let us set $2+t= \exp(\tau)$. Let $\mathcal{J}_0(\tau)$ denote the functional $\mathcal{J}_{\beta_p}(t)$ with respect to the new variable, that is, $\mathcal{J}_0(\tau)=\mathcal{J}_{\beta_p}(\exp(\tau)-2)=\mathcal{J}_{\beta_p}(t)$. Then,
\begin{align*}
\mathcal{J}_0'(\tau)& 
= (2+t)\mathcal{J}'_{\beta_p}(t),\\
\mathcal{J}_0''(\tau)& 
=(2+t)^2\mathcal{J}''_{\beta_p}(t)+(2+t)\mathcal{J}'_{\beta_p}(t).
\end{align*}
So, by using \eqref{fund rel J}, \eqref{lower bound J beta} and \eqref{lower bound J' beta}, we find that $\mathcal{J}_0(\tau)$ satisfies for  $\tau\geqslant \log 4$
\begin{equation}\label{J_0}
\begin{cases}
\mathcal{J}_0''(\tau)+2\mathcal{J}_0'(\tau)>C'_4\tau^{1-p}\mathcal{J}_0^p(\tau),\\
\mathcal{J}_0(\tau)\geqslant c_0\varepsilon^p \tau,\\
\mathcal{J}_0'(\tau)\geqslant c_0\varepsilon^p.
\end{cases}
\end{equation}


Employing \cite[Lemma 3.1 (ii)]{IS17} (see also \cite{ZH14}, where this comparison principle for ordinary differential inequalities is originally stated and proved), we get that the function $\mathcal{J}_0(\tau)$ blows up in finite time before $\tau= C\varepsilon^{-p(p-1)}$ for some constant $C>0$. 
Also, $\mathcal{J}_{\beta_p}(t)$ blows up before $t=\exp(C\varepsilon^{-p(p-1)})-2$. According to what we have said in  Remark \ref{rmk lifespans}, we have found for the lifespan $T$ of $u$ the upper bound \eqref{upper bound lifespan}. This concludes the proof of Theorem \ref{thm critical case}.
\end{proof}

\begin{rem}\label{rmk p> 2/n-sqrt delta}
Let us explain the restriction $p>\frac{2}{n-\sqrt{\delta}}$ in Theorem \ref{thm critical case}. Although it turns out as a technical condition coming from the inequality $\beta_p>\frac{\sqrt{\delta}-\mu_1+1}{2}$, in the massless case  ($\mu^2_2=0$) it is equivalent to require $\mu <\frac{n^2+n+2}{n+2}$, which is exactly the restriction on $\mu_1$ in \cite{IS17}. Furthermore, for $n\geqslant 3$ and $\delta<(n-2)^2$ this condition is always fulfilled. In particular, for high dimensions, namely for $n\geqslant 4$, we have an improvement in the range for $\delta$ for which we can prove a blow-up result in the critical case with respect to \cite{PR17vs}, where the restriction $\delta\in(0,1]$ is required. Finally, we remind that \eqref{main} is ``parabolic-like'' for $\delta\geqslant (n+1)^2$. Therefore, the restriction $\delta<(n-2)^2$ when $n\geqslant 3$ is compatible with the conjecture for \eqref{main} to be ``wave-like'' for ``small'' and nonnegative $\delta$. Similarly, in the sub-critical case, even though in Theorem \ref{sub-critical} we assume $\delta\geqslant 0$, it is clear that the result is sharp only for suitably ``small'' and nonnegative $\delta$.
\end{rem}

\begin{rem} Regarding the necessity part, in the special case $\delta=1$ the exponent $p_S(n+\mu_1)$ is proved to be really critical for $n\geqslant 3$ in the radially symmetric case in \cite{Pal18odd, Pal18even}. This shows the optimality of the range for $p$ which is obtained in this paper for suitably ``small'' and nonnegative $\delta$.
\end{rem}

\paragraph*{\bf Acknowledgment}
The first author is member of the Gruppo Nazionale per L'Analisi Matematica, la Probabilit\`{a} e le loro Applicazioni (GNAMPA) of the Instituto Nazionale di Alta Matematica (INdAM). The second author is partially supported by Zhejiang Provincial Nature Science Foundation of China under Grant No. LY18A010023.


\bibliographystyle{elsarticle-num}
\bibliography{References}

\begin{thebibliography}{99}
%
%
\bibitem{Abb15}   D'Abbicco M.
\newblock {The threshold of effective damping for semilinear wave equation.}
\newblock {\em Math. Meth. Appl. Sci. } {\bf 38} (2015), no. 6, 1032-1045.
%
\bibitem{DabbLuc15} D'Abbicco M,  Lucente S. 
\newblock{ NLWE with a special scale invariant damping in odd space dimension.}
\newblock{\em Discrete Contin. Dyn. Syst.} 2015,
\newblock{ Dynamical systems, differential equations and applications.} 10th AIMS Conference. Suppl., 312-319.
%
%
\bibitem{DLR15}  D'Abbicco M, Lucente S,  Reissig M. 
\newblock { A shift in the Strauss exponent for semi-linear wave equations with a not effective damping.}
\newblock {\em J. Diff. Equations } {\bf 259} (2015), no. 10, 5040-5073.
%
\bibitem{DabbPal18} D'Abbicco M,  Palmieri A.
\newblock{$L^p-L^q$ estimates on the conjugate line for semilinear critical dissipative Klein-Gordon equations.}
\newblock{ Preprint.}
%
\bibitem{Erdelyi} Erdelyi A, Magnus W, Oberhettinger F, Tricomi, FG. \newblock{\em Higher Transcendental Functions, vol. 2.} McGraw-Hill, New York (1953)
%
%
%
%
\bibitem{GLS97} Georgiev V, Lindblad H, Sogge C D. 
\newblock {Weighted Strichartz estimates and global existence for semi-linear wave equations.}
\newblock {\em Amer. J. Math.} {\bf 119} (1997), no. 6, 1291-1319.
%
%
\bibitem{Gla81-g} Glassey R T.
\newblock {Existence in the large for $\square u = F(u)$ in two space dimensions.}
\newblock {\em Math Z.} {\bf 178} (1981), no. 2, 233-261.
%
%
\bibitem{Gla81-b} Glassey R T. 
\newblock {Finite-time blow-up for solutions of nonlinear wave equations.}
\newblock {\em Math Z.} {\bf 177} (1981), no. 3, 323-340.
%
\bibitem{IS17} Ikeda M, Sobajima M.
\newblock{Life-span of solutions to semilinear wave equation with time-dependent critical damping for specially localized initial data.}
\newblock{\em Math. Ann.} (2018)   https://doi.org/10.1007/s00208-018-1664-1. 
%
%
\bibitem{Joh79} John F.
\newblock {Blow-up of solutions of nonlinear wave equations in three space dimensions.}
\newblock{\em Manuscripta Math.} {\bf 28} (1979), no. 1-3, 235-268.
%
%
%
\bibitem{LST18} Lai N A, Schiavone N M, Takamura H.
\newblock{Wave-like blow-up for semilinear wave equations with scattering damping and negative mass.} \newblock{ Preprint, arXiv:1804.11073v2. }
%
\bibitem{LT18Scatt} Lai N A, Takamura H.
\newblock{Blow-up for semilinear damped wave equations with subcritical exponent in the scattering case.} \newblock{\em Nonlinear Anal.} {\bf 168} (2018), 222-237.
%
\bibitem{LT18ComNon} Lai N A, Takamura H.
\newblock{Nonexistence of global solutions of wave equations with weak time-dependent damping and combined nonlinearity.} \newblock{ Preprint, arXiv:1802.10273. }
%
\bibitem{LTW17} Lai N A, Takamura H, Wakasa K. \newblock{ Blow-up for semilinear wave equations with the scale invariant damping and super-Fujita exponent.}
\newblock{\em J. Differential Equations} {\bf 263} (2017), no. 9, 5377-5394.
%
%
\bibitem{LS96} Lindblad H, Sogge C D.
\newblock{ Long-time existence for small amplitude semilinear wave
equations.}
\newblock{\em Amer. J. Math.} {\bf 118} (1996), no. 5, 1047-1135.
%
\bibitem{NPR16} do Nascimento W Nunes,  Palmieri A,  Reissig M. \newblock{ Semi-linear wave models with power non-linearity and scale-invariant time-dependent mass and dissipation.}
\newblock{\em Math. Nachr.}  {\bf 290} (2017), no. 11-12, 1779-1805.
%
\bibitem{OLBC10} F. W. J. Olver, D. W. Lozier, R. F. Boisvert and C. W. Clark editors.
\newblock{\em  NIST handbook of mathematical functions.}
\newblock{ U.S. Department of Commerce, Washington, D.C. 2010.}
%
\bibitem{Pal17}  Palmieri A. {\newblock{Global existence of solutions for semi-linear wave equation with scale-invariant damping and mass in exponentially weighted spaces.}
\newblock{\em J. Math. Anal. Appl.} {\bf 461} (2018), no. 2, 1215-1240.}
%
\bibitem{Pal18odd}  Palmieri A.
 \newblock{Global existence results for a semilinear wave equation with scale-invariant damping and mass in odd space dimension}. \newblock{ Preprint.}
%
\bibitem{Pal18even}  Palmieri A.
\newblock{A global existence result for a semilinear wave equation with scale-invariant damping and mass in even space dimension}. \newblock{ Preprint.}
%
\bibitem{PalRei17}   Palmieri A, Reissig M.
\newblock{Semi-linear wave models with power non-linearity and
scale-invariant time-dependent mass and dissipation, II.}
\newblock{\em Math. Nachr.} (2017) 1-34;  doi: 10.1002/mana.201700144.
%
\bibitem{PR17vs} Palmieri A, Reissig M.
\newblock{A competition between Fujita and Strauss type exponents for blow-up of semi-linear wave equations with scale-invariant damping and mass.} \newblock{Accepted for publication in \em J. Diff. Equations.} (2018). 
%
%
\bibitem{Sch85} Schaeffer J.
\newblock {The equation $u_{tt}-\Delta u = |u|^p$ for the critical value of $p$.}
\newblock {\em Proc. Roy. Soc. Edinburgh Sect. A.} {\bf 101} (1985), no. 1-2, 31-44.
%
%
\bibitem{Sid84} Sideris T C. 
\newblock { Nonexistence of global solutions to semilinear wave equations in high dimensions.}
\newblock {\em  J. Diff. Equations} {\bf 52} (1984), no. 3, 378-406.
%
%
\bibitem{TL1709} Tu Z, Lin J. 
\newblock{A note on the blowup of scale invariant damping wave equation with sub-Strauss exponent.} 
\newblock{ Preprint arXiv:1709.00866v2.} 
%
%
\bibitem{TL1711} Tu Z, Lin J.
\newblock{ Life-Span of Semilinear Wave Equations with Scale-invariant Damping: Critical Strauss Exponent Case.}. \newblock{ Preprint arXiv:1711.00223.}
%
%
\bibitem{Wakasugi14} Wakasugi Y.
\newblock{Critical exponent for the semilinear wave equation with scale invariant damping.}
\newblock{\em Fourier Analysis}, 375-390,
\newblock{Trends Math.}, \newblock{\em Birkh\"auser/Springer}, \newblock{\em Cham}, (2014).
%
%
\bibitem{Wirt04} Wirth J.
\newblock {Solution representations for a wave equation
with weak dissipation.}
\newblock {\em  Math. Meth. Appl. Sci.} {\bf 27} (2004), no. 1, 101-124.
%
%
\bibitem{YZ06} Yordanov B T, Zhang Q S.
\newblock{ Finite time blow up for critical wave equations in high dimensions.}
\newblock{\em J. Func. Anal.} {\bf 231} (2006), no. 2, 361-374.
%
%
\bibitem{Zhou95} Zhou Y.
\newblock { Cauchy problem for semi-linear wave equations in four space dimensions with small initial data.}
\newblock {\em J. Diff. Equations} {\bf 8} (1995), 135-144.
%
%
\bibitem{Zhou07} Zhou Y.
\newblock{ Blow up of solutions to semilinear wave equations with critical exponent in high dimensions.} 
\newblock{\em Chin. Ann. Math. Ser. B} {\bf 28} (2007), no. 2, 205-212. 
%
\bibitem{ZH14} Zhou Y, Han W.
\newblock{ Life-span of solutions to critical semilinear wave equations.}\newblock{\em Comm. Partial Differential Equations} {\bf 39} (2014), no. 3, 439-451. 
%
\end{thebibliography}
 
\end{document}